\documentclass[12pt]{article}
\usepackage[centertags,intlimits]{amsmath}
\usepackage{amsfonts}
\usepackage{amsthm}
\usepackage{graphics}
\usepackage{color}
\usepackage{graphicx}
\usepackage{amssymb}
\usepackage{hyperref}
\allowdisplaybreaks[2]
\numberwithin{equation}{section}
\newtheorem{theorem}{Theorem}[section]
\newtheorem{lemma}{Lemma}[section]

\newtheorem{remark}{Remark}[section]
\newtheorem{conjecture}{Conjecture}

\providecommand{\abs}[1]{\lvert #1\rvert}

\newcounter{neweqn}
\newcommand{\beq}[1]{\addtocounter{neweqn}{1}\begin{equation}\label{#1}}
\newcommand{\eeq}{\end{equation}}

\newcommand{\nc}{\newcommand}

\newcommand{\beray}[1]{\addtocounter{neweqn}{1}\begin{eqnarray}  \label{#1}}

\nc{\vb}{\mathbf{v}}
\nc{\bu}{\mathbf{u}}
\nc{\bv}{\mathbf{v}}
\nc{\bq}{\mathbf{q}}
\nc{\bd}{\mathbf{d}}
\nc{\bb}{\mathbf{b}}
\nc{\bc}{\mathbf{c}}
\nc{\bi}{\mathbf{i}}
\nc{\bfr}{\mathbf{r}}
\nc{\bP}{\mathbf{P}}
\nc{\bQ}{\mathbf{Q}}
\nc{\bbC}{\mathbb C}
\nc{\D}{\mathbb D}
\nc{\F}{\mathbf F}
\nc{\bbS}{\mathbb S}
\nc{\bE}{\mathbf E}
\nc{\br}{\bigr}
\nc{\bl}{\bigl}
\nc{\Bl}{\Bigl}
\nc{\Br}{\Bigr}
\nc{\ind}[1]{\,\mathbf{1}_{\{#1\}}\,
}

\newcommand{\by}{{\bf y}}
\newcommand{\bx}{{\bf x}}

\newcommand{\R}{{\mathbb{R}}}
\renewcommand{\P}{\mathsf{P}}

\newcommand{\Z}{\mathbb{Z}}
\newcommand{\N}{\mathbb{N}}

\newcommand{\eps}{\varepsilon}
\newcommand{\E}{\mathsf{E}}

\newcommand{\A}{{\bf A}}

\newcommand{\be}{{\bf e}}

\newcommand{\Xn}{X^{n}}
\newcommand{\Qn}{Q^{n}}
\newcommand{\Mn}{M^{n}}

\begin{document}

\title{A diffusion limit for Markov chains with
 log-linear interaction on a graph }

\author{
Anatolii Puhalskii\footnote{Institute for Problems in Information Transmission, Moscow, Russia.
 Email address: puhalski@iitp.ru
 }\, and 
Vadim Shcherbakov\footnote{
 Royal Holloway University of London, Egham, UK.
 Email address: vadim.shcherbakov@rhul.ac.uk
}
}

\date{}
\maketitle


\begin{abstract}
{\small
In this paper we establish a  diffusion limit  for  
a multivariate continuous time   Markov chain 
 whose 
components  are indexed
by vertices of a finite graph.
The components  take values in a common finite  set of non-negative integers 
and evolve subject to  a graph based log-linear  interaction.  
We show that if the set of common  values of the components expands to the set of all 
non-negative integers, then a time-scaled and normalised
version of the Markov chain converges to 
a system of interacting Ornstein-Uhlenbeck processes reflected at the origin.
This limit  is akin to heavy traffic limits 
in queueing (and our model can be naturally 
interpreted as a queueing model).  
 Our 
proof draws on developments  in queueing theory
    and relies on  martingale methods.

}
\end{abstract}


\section{Introduction}
\label{intro}

This paper concerns a probabilistic model 
that is stated in terms of a multivariate 
continuous time Markov chain (CTMC), whose components  are 
indexed by vertices of a finite undirected graph 
and take values in a common finite set of non-negative integers.
Components   evolve subject to a nearest-neighbour interaction, where 
the neighbourhood relationship is induced by the underlying graph.
The CTMC can be interpreted as a queueing system with interaction and is also related to interacting
spin systems of statistical physics. 
 
We are interested in the asymptotic regime where
the set of possible values 
of the components  expands to the set of all non-negative integers, while 
the interaction weakens. 
We show  that a
time-scaled and normalised version of the chain converges in distribution to 
a continuous--path process  which can be 
 interpreted as a collection  of interacting Ornstein-Uhlenbeck processes with 
  reflection.
  It should be noted that the limit regime that we consider 
   is reminiscent of the  heavy traffic 
scaling  in  queueing theory. 
Our 
proof draws on developments  in queueing theory
    and relies on  martingale methods, widely used for studying
    queueing systems under  heavy traffic conditions.

\section{The model}
\label{model}

It is assumed that all random variables are defined on a common
probability space endowed with probability measure $\P$.
 Expectation with respect to $\P$ will be denoted by $\E$.
Let  $G=(V,E)$ be a finite undirected graph with the set of vertices $V$ and the set 
of edges $E$. 
If vertices $v$ and $u$ are adjacent, we call them neighbours and write $v\sim u$.
By definition, a vertex is not a neighbour of itself.
Let $\A=(a_{vu})$ represent
 the  adjacency matrix of the graph $G=(V,E)$, that is a symmetric matrix
such that $a_{vu}=a_{uv}=0$, if $u\nsim v$ and  $a_{vu}=a_{uv}=1$, if 
 $u\sim v$, where $v,u\in V$.

Given an integer $N\geq 1$, consider  a CTMC
 $$Q(t)=(Q_{v}(t), v\in V)\in S_{N, V}:=\{0,1,\ldots,N\}^{V}$$
  (i.e.  $Q_{v}(t)\in\{0,1,\ldots,N\}$)
and 
  with the transition rates 
$r(\bx,\by)$, $\bx,\by\in S_{N,V}$, given by 
 \begin{equation}
\label{rates1}
r(\bx,\by)=\begin{cases}
\lambda_v(\bx),& 
\mbox{ for }\by=\bx+\be_v \mbox{ and } \bx=
(x_u,\, u\in V): x_v<N,\\
1,& 
\mbox{ for } \by=\bx-\be_v \mbox{ and } \bx=
(x_u,\, u\in V): x_v>0,\\
0,& \text{otherwise},
\end{cases}
\end{equation}
where 
\begin{equation}
\label{lambda}
\lambda_v(\bx)=e^{\alpha x_v+\beta(\A\bx)_v}
=e^{\alpha x_v+\beta\sum\limits_{u: u\sim v}x_u},
\end{equation}
$\alpha$ and $\beta$  are given constants,   and
$\be_v\in \R^V$ is the vector, the  $v-$th coordinate of which 
is equal to $1$,  and all other coordinates are zeros.

The model is motivated by the study of interactions in multicomponent systems. In particular, it is designed to 
capture a common feature of many real-life phenomena: in the absence of interaction, each component follows a 
simple process, whereas the presence of interaction can significantly alter both individual and collective 
behaviours. The log-linear rates provide a flexible and technically convenient framework for modelling different 
types of interaction.
Indeed,  if $\beta=0$, then the CTMC $Q(t)$  is just 
 a collection of i.i.d.  one-dimensional birth-and-death  Markov chains
  $(Q_v(t),\, t\geq 0)$, $v\in V$, with exponential birth rates and unit death rates.
 Each of these processes  can be also regarded  as  a non-homogeneous simple 
random walk  on the set of integers $\{0,1,...,N\}$ 
with reflection at both $0$ and $N$. 
If $\beta\neq 0$, then  the CTMC $Q(t)$
can be interpreted as a system of the aforementioned one-dimensional 
processes evolving subject to the 
 interaction induced by the parameter $\beta$.  
 If $\beta>0$, then the interaction is cooperative
in the sense that positive components increase the birth rates of their neighbours.
Vice versa, if $\beta<0$, then  the interaction is competitive in the sense that 
components obstruct the growth of each other.
The CTMC $Q(t)$ is reversible, and its stationary distribution is available in closed form 
(see Section~\ref{stationary-dist}).

  This model  can be considered on an infinite graph with a bounded vertex degree (e.g. on the lattice) as well, in which case 
  it is related to interacting particle systems 
  such as the 
Richardson model (\cite{Richardson}) and the contact process (\cite{Liggett}).
  The zero death rate case
   is related  to a class of spatial growth models with nearest-neighbour interaction introduced
  in~\cite{Harris}.

A special case of the  model with $\alpha=\beta$ was introduced in~\cite{Yamb}.    
The model in its current form, with arbitrary $\alpha$ and $\beta$, 
is  a state-space constrained  version of the 
 countable CTMC (i.e. the case when ``$N=\infty$'')
 introduced in \cite{VS15} and later studied in greater detail
  in \cite{JVS}.
  The countable CTMC exhibits all possible modes of asymptotic
   behaviour of a countable CTMC: null or positive recurrence, and non-explosive or explosive transience, depending on the parameters $\alpha$ and $\beta$ and on the structure of the underlying graph. For more details, see \cite[Theorems 2.3 and 6.1]{JVS}. 
    Here we just briefly describe another notable feature of the countable CTMC, 
    namely the phase transition in the long-term behaviour of the process. In particular, for $\alpha<0$,
    define
    \begin{equation}
\label{beta}
\beta_{cr}:=-\alpha/\nu(G),
\end{equation}
 where $\nu(G)$ denotes  the principal 
eigenvalue of the graph $G$, i.e. the largest  eigenvalue of the adjacency matrix of the graph.
 Then the CTMC is
\begin{itemize}
\item positive recurrent, if $\beta<\beta_{cr}$,
\item    non-explosive transient, if   $\beta=\beta_{cr}$, and 
\item  explosive transient, if $\beta>\beta_{cr}$.
\end{itemize}

The CTMC $Q(t)$ is also related to models of interacting spins in statistical physics. 
For instance, when $N=1$, the stationary distribution of the Markov chain 
coincides with a special case of the Ising model on a finite graph
 (see Section~\ref{stationary-dist}). In statistical physics,
  the primary interest lies in the behaviour of such models as the
   underlying graph grows indefinitely, 
   particularly in relation to the occurrence of phase transitions.
We study  the model 
 in the asymptotic regime where the graph $G$ is fixed,  $N$ tends to infinity, and the appropriately scaled 
 interaction vanishes.
 
 More specifically, we consider a sequence of CTMC 
  $\Qn(t)=(\Qn_v(t),\, v\in V)$, $n\in \N=\{1,2,\ldots\}$,  with   transition rates 
  $r_n(\bx, \by)=r(\bx/n, \by/n)$,
$\bx, \by\in S_{N_n, V}$, 
where $r(\cdot, \cdot)$ are the transition rates defined
in~\eqref{rates1}, so that $\alpha$ and $\beta$ are 
effectively
replaced with $\alpha/n$ and $\beta/n$, respectively.
We  assume that $N=N_n\to \infty$ in such a way that
$N_n/\sqrt{n}\to\infty$, while $N_n/n\to 0$, as $n\to \infty$, and
  show that 
   the process $(Q^n(nt)/\sqrt{n}),\,  t\in\R_+)$ converges in
   distribution to a multivariate continuous path process.
    This  limit  process 
 is a collection of interacting Ornstein-Uhlenbeck 
  (OU) processes reflected at the origin.

A key insight is to view the  CTMC $Q^n(t)=(Q^n_v(t),\, v\in V)$
  as a collection of probabilistically independent
single--server exponential queues with finite buffers.
  The queues are associated
with the vertices,  the component  $Q_v^n(t)$ representing the queue length at vertex $v$.
 The   transitions $Q^n_v\to Q^n_v+1$ correspond to customer arrivals
and the  transitions $Q^n_v\to Q^n_v-1$ correspond to customer departures.
The arrival rates are equal to $\lambda_v(Q^n(t))$\, so that they
 depend on the states of
the other queues, whereas the service rates are equal to $1$. Since
$\lambda_v( 0) =1 $,  the arrival and nominal service
rates match when there are no customers present, so,
the queues are  critically loaded, where 
$ 0$ represents the origin of  $\R_+^V$.
 This observation enables us to apply  techniques developed in \cite{Lee} to obtain
      diffusion approximation
results for critically loaded queueing
networks with state dependent rates.
In addition,
  the martingale methods developed for the study of
 exponential queueing networks are brought to bear on the present setup
\cite{Krichagina,PangTalrejaWhitt}.

Here's how this paper is organised. In Section~\ref{main}
we  recall the basics of the Skorohod reflection mapping, 
formally define the limit process and state the main result. The proof of the main result is given 
in Section~\ref{proof}. In Section~\ref{stationary-dist} we discuss  the model's stationary distribution 
 and its diffusion limit.  Finally, in Section~\ref{problem} we 
 state an open problem  concerning the long term behaviour of 
the limit process.

\section{Skorohod reflection and the main result}
\label{main}

Reviewing  the  basics of the Skorohod reflection mapping on
$\R_+$ is in order. 
Given a real valued rightcontinuous function 
$\psi=(\psi(t), \, t\geq 0)$
with lefthand limits such that $\psi(0)\ge0$,
there exists a uniquely specified real--valued rightcontinuous 
function $\Gamma(\psi)=(\Gamma(\psi)(t),\, t\ge0)$ with lefthand limits  
such that $\Gamma(\psi)(t)\ge0$, the function
$\phi=\Gamma(\psi)-\psi$
 is nondecreasing,  and 
$$\phi(t)=\int_0^t{\bf 1}_{\{\Gamma(\psi)(s)=0\}}\,d\phi(s).$$
This result for the case of  $\psi$ being  continuous  was first
obtained in~\cite{Skorohod}, see also \cite{ike}.  The proof in \cite{ike} also
applies when $\psi$ is rightcontinuous with lefthand limits.
We will say that the pair $(\Gamma(\psi),\phi)$ is a solution  of
 the Skorohod reflection problem on
$\R_+ $ associated with $\psi$.
Moreover, the following explicit representation holds (e.g. see~\cite{ike})
\begin{equation}
\label{Gamma-def}
\Gamma(\psi)(t)=\psi(t)- 0\wedge \inf\limits_{s\in[0, t]}\psi(s).
\end{equation}
It follows that the map $\psi\to\Gamma(\psi)$ from
$\D(\R_+,\R)$ to $\D(\R_+,\R)$ is
Lipschitz continuous for
the locally uniform metric.
 The following majorisation property is useful.
The proof is a direct
consequence of \eqref{Gamma-def}.
\begin{lemma}
  \label{le:maj}
Suppose that $\psi_1=(\psi_1(t),\,  t\ge0)$ and
$\psi_2=(\psi_2(t),\, t\ge0)$ are rightcontinuous functions with
lefthand limits such that $\psi_1(0)\ge\psi_2(0)\ge0$. If $\psi_1$
strongly majorises $\psi_2$ in the sense that the function 
$(\psi_1(t)-\psi_2(t),\,t\ge0)$ is nondecreasing, then 
$\Gamma(\psi_1)(t)\ge \Gamma(\psi_2)(t),\, t\ge0$.
\end{lemma}

Let $B=(B_v,\, v\in V)$, where $B_v=(B_v(t),\, t\geq 0)$, 
represent  a collection of
 independent one-dimensional standard Brownian motions
 indexed by vertices of the graph $G=(V,E)$.
Let  $X_v=(X_v(t),t\ge0),\,  v\in V,$ be  
continuous path nonnegative processes that follow the equations
\begin{align}
  \label{eq:20'}
dX_v(t)&=\big(\alpha X_v(t)+\beta(\A X(t))_v\big)\,dt + 
\sqrt{2}\,dB_v(t) +d\phi_{v}(t),\quad v\in V,
\end{align}
with some initial conditions $X_v(0)\geq 0$, where   
$\phi_v=(\phi_v(t),\, t\geq 0),\, v\in V,$ are  nondecreasing 
continuous path processes such that 
    \begin{align*}
  \phi_v(t)=
\int_{0}^{t}{\bf 1}_{\{X_v(s)=0\}}d\phi_v(s)
\end{align*}
and $X(t)=(X_v(t),\,v\in V)$\,.
Thus, 
$X_v=\Gamma(Y_v),\, v\in V,  $
 where 
$Y_v=(Y_v(t),\, t\ge0)$, $Y_v(0)=X_v(0)$, and 
\begin{equation}
\label{Yv}
dY_v(t)=\big(\alpha X_v(t)+\beta(\A X(t))_v\big)\,dt + \sqrt{2}\,dB_v(t),\quad v\in V\,.
\end{equation}
The distribution of $X=(X_v,\, v\in V)$ solves a diffusion martingale problem
on normal  reflection in $\R_+^{V}$,\,  cf.  \cite{AnuLip}. 
The results of   \cite{AnuLip}  imply that the process $X$ is well defined.

Let $\mathbb D(\mathbb R_+,\mathbb R^{V})$ represent the Skorohod
space  of $\mathbb R^{V}$--valued  rightcontinuous functions on $\mathbb
R_+$ with lefthand limits. It is endowed with the Skorohod--Lindvall metric, 
which renders
$\mathbb D(\mathbb R_+,\mathbb R^{V})$  a complete separable metric space, see, e.g.,
\cite{jacshir}. 

Recall that
  $\Qn(t)=(\Qn_v(t),\, v\in V)$  is the CTMC 
   with   transition rates $r_n(\bx, \by)=r(\bx/n, \by/n)$,
$\bx, \by\in S_{N_n, V}$, 
where $r(\cdot, \cdot)$ are the transition rates defined
in~\eqref{rates1}.
The process $\Qn=(\Qn(t),t\ge0)$  is regarded as a random element of $\mathbb D(\mathbb R_+,\mathbb
R^{{V}})$ equipped with the Borel $\sigma$-algebra.
Let  
\begin{equation}
\label{Xn}
X^n_v(t)=\Qn_v(nt)/\sqrt{n}\,,\;\;
X^n(t)=\Qn(nt)/\sqrt{n}\quad\text{and}\quad X^n=(X^n(t),\, t\ge0).
\end{equation}
Theorem~\ref{the:1} below is the main result of the paper.

\begin{theorem}
\label{the:1}
Suppose that  $N_n\to \infty$, $N_n/\sqrt{n}\to\infty$ and $N_n/n\to 0$, as $n\to \infty$.
Suppose also  that  the initial condition $X^n(0)$ is deterministic, and 
 $X^n(0)\to X(0)=(X_v(0), \, v\in V)\in\R_+^{V}$ component-wise, as $n\to \infty$.
Then the  process
 $X^n$ converges in distribution  in 
$\mathbb D(\mathbb \R_+,\mathbb R^{{V}})$ to the
process $X=(X(t),\, t\geq 0)$.
\end{theorem}

\section{The proof of Theorem~\ref{the:1}}
\label{proof}

In this section  we denote by $c_i$ various positive constants whose 
specific  values are immaterial 
for the proof of the theorem.

The model being Markovian implies that the following representation holds
(see \cite[Chapter 6, Section 4]{Kurtz} for more detail):
\begin{align}
\label{P1}
Q_v^n(t)&=Q_v^n(0)+\Pi^n_{v,1}\bigg(\int_0^t
\lambda_v\Big(\frac{Q^n(s)}{n}\Big)
{\bf 1}_{\{Q_v^n(s)<N_n\}}ds\bigg)-
\Pi_{v,2}^n\bigg(\int_0^t{\bf 1}_{\{Q^n_v(s)>0\}}ds\bigg),
\end{align}
where   $\Pi^n_{v,1}=(\Pi^n_{v,1}(t),\, t\geq 0)$ and $
\Pi^n_{v,2}=
(\Pi^n_{v,2}(t),\,t\geq 0)$,\, $v\in  V$,  are 
independent unit rate  Poisson processes.
Via an elementary algebraic manipulation, \eqref{P1} can be written as
\begin{align}
\label{P2}
Q^n_v(t)&=Q^n_v(0)+\int_0^t
\bigg(\lambda_v\Big(\frac{Q^n(s)}{n}\Big)-1\bigg)ds
\notag+M^n_{v,1}(t)-M^n_{v,2}(t)\\
&+\int_0^t{\bf 1}_{\{Q^n_v(s)=0\}}ds
-\int_0^t\lambda_v\Big(\frac{Q^n(s)}{n}\Big){\bf 1}_{\{Q^n_v(s)=N_n\}}ds,
\end{align}
where 
\begin{align}
\label{M1}
M^n_{v,1}(t)&=\Pi^n_{v,1}\bigg(\int_0^t
\lambda_v\Big(\frac{Q^n(s)}{n}\Big){\bf 1}_{\{Q^n_v(s)<N_n\}}ds
\bigg)-
\int_0^t
\lambda_v\Big(\frac{Q^n(s)}{n}\Big){\bf 1}_{\{Q^n_v(s)<N_n\}}ds\intertext{ and }
\label{M2}
M^n_{v,2}(t)&=\Pi^n_{v,2}\bigg(\int_0^t
{\bf 1}_{\{Q^n_v(s)>0\}}ds\bigg)-
\int_0^t
{\bf 1}_{\{Q^n_v(s)>0\}}ds.
\end{align}
An application of the results of \cite[Chapter 6, Section 4]{Kurtz}
to \eqref{M1} and \eqref{M2} implies that
the processes  $M^n_{v,1}=(M^n_{v,1}(t)\,,t\ge0)$ and 
$M^n_{v,2}=(M^n_{v,2}(t),t\ge0)$ 
are orthogonal locally square integrable martingales relative to the
natural filtration with the 
predictable quadratic variation (angle-bracket)  processes
$\langle M^n_{v,1}\rangle=(\langle M^n_{v,1}\rangle(t),t\ge0)$ given by 
\begin{equation}
\label{M1sq}
\langle M^n_{v,1}\rangle(t)=
\int_0^t
\lambda_v\Big(\frac{Q^n(s)}{n}\Big){\bf 1}_{\{Q^n_v(s)<N_n\}}ds
\end{equation}
and 
\begin{equation}
\label{M2sq}
\langle M^n_{v,2}\rangle(t)=\int_0^t
{\bf 1}_{\{Q^n_v(s)>0\}}ds,
\end{equation}
respectively.

By~\eqref{P2},
\begin{equation}
\label{Xnv}
\Xn_v(t)=\Xn_v(0)+D^n_v(t)
+M^n_v(t)+\phi^n_{v}(t)-\varphi_v^n(t),
  \end{equation}
where
\begin{align}
  \label{eq:12}
  D^n_v(t)&=
\sqrt{n}\int\limits_0^{t}
\left(\lambda_v\Big(\frac{\Xn(s)}{\sqrt{n}}\Big)-1\right)ds,\\
\notag 
M^n_v(t)&=\frac{1}{\sqrt{n}}\,\Mn_{v,1}(nt)-\frac{1}{\sqrt{n}}\,\Mn_{v,2}(nt),\\
\label{eq:12c}
\phi^n_{v}(t)&=\sqrt{n}\int_0^{t}
{\bf 1}_{\{\Xn_v(s)=0\}}\,ds,\\
\label{varphi}
\varphi^n_v(t)&=\sqrt{n}\int_0^t
\lambda_v\Big(\frac{\Xn(s)}{\sqrt{n}}\Big)
{\bf 1}_{\{\Xn_v(s)=N_n/\sqrt{n}\}}\,ds.
\end{align}
The processes $M_v^n=(M_v^n(t),\, t\ge0), \, v\in V$,
are orthogonal locally square integrable
martingales with predictable quadratic variation processes
$\langle M^n_v\rangle=(\langle M^n_v\rangle(t), t\ge0)$ given by  (see
\eqref{M1sq} and \eqref{M2sq})
\begin{equation}
  \label{eq:1}
  \langle M^n_v\rangle(t)=
\int_0^t\lambda_v\Big(\frac{X^n(s)}{\sqrt{n}}\Big)
{\bf 1}_{\{X^n_v(s)<N_n/\sqrt{n}\}}ds+
\int_0^t
{\bf 1}_{\{X^n_v(s)>0\}}ds.
\end{equation}
Hence, with $M^n(t)=(M^n_v(t),\, v\in V)$,\, the process 
$M^n=(M^n(t),t\ge0)$ is an $\R^V$--valued
locally square integrable  martingale with predictable quadratic
variation process  $\langle M^n\rangle=(\langle
M^n\rangle_{v,v'},\,v,v'\in V)$\,, where 
$\langle M^n\rangle_{v,v'}=(\langle M^n\rangle_{v,v'}(t),\, t\ge0)$  and
\begin{equation}
  \label{eq:5}
  \langle M^n\rangle_{v,v'}(t)=
\langle M^n_{v}\rangle(t)\mathbf 1_{\{v=v'\}}.
\end{equation}

 As the next step, we  establish tightness properties for
 processes $M^n$, 
$D^n$ and $X^n$ as random elements of the
associated Skorohod spaces.

\begin{lemma}
\label{le:X^n_bound}For all $t>0$ and $v\in V$,
  \begin{equation}
  \label{eq:9v}
\lim_{A\to\infty}\limsup_{n\to\infty}  \P\Big(\sup_{s\le t} |M^n_v(s)|
\ge A\Big)=0
\end{equation}
and \begin{equation}
  \label{eq:5v}
  \lim_{A\to\infty}\limsup_{n\to\infty}
\P\Big(\sup_{s\le t} X^n_v(s)\ge A\Big)=0\,.
\end{equation}
\end{lemma}
\begin{proof}
Recall first that 
\[
0\leq \frac{\Xn_v(s)}{\sqrt{n}}=\frac{\Qn_v(s)}{n}\leq \frac{N_n}{n},
\]
which implies by \eqref{lambda} that
 \begin{equation}
\label{lambda<}
\lambda_v\Big(\frac{X^n(s)}{\sqrt{n}}\Big)\leq
e^{(|\alpha|\Xn_v(s)+|\beta|(\A X^n(s))_v)/\sqrt{n}}\leq 
e^{c_1\frac{N_n}{n}}\to 1,
 \end{equation}
 as $n\to \infty$,
   and 
 \begin{equation}
  \label{Tay1}
  \sqrt{n}\bigg|\lambda_v\Big(\frac{X^n(s)}{\sqrt{n}}\Big)-1\bigg|
\le \big(|\alpha|\Xn_v(s)+|\beta|(\A X^n(s))_v\big)e^{c_2N_n/n},
 \end{equation}
 for 
  all sufficiently large $n$.

Further, by Doob's inequality,~\eqref{eq:1} and~\eqref{lambda<}, for all $A>0$\,,
\begin{equation}
  \label{eq:4v}
  \P\Big(\sup_{s\le t} |M^n_v(s)|\ge A\Big)
  \le
\frac{1}{A^2}\,\E\big(M^n_v(t)^2\big)
\le\frac{1}{A^2}\,\E\big(\langle M^n_v\rangle(t)\big)
\le \frac{1}{A^2}\big(e^{c_1\frac{N_n}{n}}+1\big)t,
\end{equation}
so that \eqref{eq:9v} holds.

 Let us show~\eqref{eq:5v}. To this end, 
 note that equation~\eqref{Xnv}  for $X^n_v=(X_v^n(t),\, t\geq 0)$ can be written as follows
\begin{equation}
\label{Xnv1}
X_v^n=\Gamma\left(\psi^n_v\right),
\end{equation}
where 
$\psi^n_v=(\Xn_v(0)+D^n_v(t) +M^n_v(t)-\varphi^n_{v}(t),\, t\geq 0)$.
Observe that the process $\varphi^n_v$ (see~\eqref{varphi}) is non-decreasing, so that 
$\widetilde\psi^n_v:=(\Xn_v(0)+D^n_v(t)+M^n_v(t),\, t\geq 0)$ strongly 
majorises $\psi^n_v$. Therefore, applying Lemma~\ref{le:maj}, we obtain that 
$X_v^n(t)\leq \Gamma(\widetilde\psi^n_v)(t)\text{ for all } t\geq0,
$ which implies, since the map $\Gamma$ is Lipschitz-continuous, that 
\begin{equation}
\label{eq:2va}
 X^n_v(t)
\le
K\Big(X^n_v(0)+\sup_{s\le t}|D^n_v(s)|+\sup_{s\le t}|M^n_v(s)|\Big),
\end{equation}
with  some $K>0$ for all $t\geq 0$.

By~\eqref{eq:12} and~\eqref{Tay1},
\begin{equation}
\label{Dv-bound}
|D^n_v(s)|\leq
\sqrt{n}\int\limits_0^{s}
\bigg|\lambda_v\Big(\frac{\Xn(s')}{\sqrt{n}}\Big)-1\bigg|ds'
\leq
e^{c_2N_n/n}
\int\limits_0^{s}
\big(|\alpha|\Xn_v(s')+|\beta|(\A X^n(s'))_v)\big)ds' 
\end{equation}
for  all $s>0$, and, hence, 
\begin{equation}
\label{Dv-sup}
\sup_{s\le t}|D^n_v(s)|\leq e^{c_2N_n/n}
\int\limits_0^{t}
\big(|\alpha|\Xn_v(s)+|\beta|(\A X^n(s))_v\big)ds
\end{equation}
for all $t>0$.

Equations \eqref{eq:2va} and ~\eqref{Dv-sup} 
yield the bound
\begin{equation}
\label{eq:2vb}
 X^n_v(s)
\le c_3
\bigg(X^n_v(0)+\int_0^s\big(
|\alpha|X^n_v(s')+|\beta|(\A X^n(s'))_v\big)\,ds'
+\sup_{s'\le s}|M^n_v(s')|\bigg).
\end{equation}
Further, 
observe that the following identity 
holds
\[
\sum\limits_{v\in V}(\A\bx)_v=\sum\limits_{v\in V}
\Big(\sum_{u:u\sim v}x_u\Big)
=\sum\limits_{v\in V}d_vx_v\quad\text{for}\quad\bx=(x_v,\, v\in V)\in \R^{V},
\]
where  $d_v$ is the  degree of  vertex $v\in V$ 
(i.e., $d_v$ is the number of neighbours of $v$).
Therefore, 
\begin{equation}
\label{graph}
\sum\limits_{v\in V}(\A\bx)_v
\leq \Big(\max_{v\in V}d_v\Big)\sum\limits_{v\in V}x_v
\quad\text{for}\quad\bx=(x_v,\, v\in V)\in \R_+^{V}.
\end{equation}
Summing up equations~\eqref{eq:2vb} and using~\eqref{graph}, we get 
that  
\begin{align*}
\sum\limits_{v\in V} X^n_v(s)
&\le
c_4\bigg(\sum_{v\in V}X^n_v(0)
+\int_0^s\sum_{v\in V} X^n_v(s')\,ds'
+\sum_{v\in V}\sup_{s'\le s}|M^n_v(s')|\bigg).
\end{align*} 

By the Gronwall--Bellman inequality,
\begin{equation}
  \label{eq:3v}
 \sup_{s\le t}\sum\limits_{v\in V} X^n_v(s)\le
c_5e^{c_5t}\sum\limits_{v\in V}
\Big(X^n_v(0)+\sup_{s\le t}|M^n_v(s)|\Big)\,.
\end{equation}
Finally, combining \eqref{eq:9v} and \eqref{eq:3v} with the facts that
(due to  $\Xn_v(t)$ being nonnegative)
$X^n_v(s)\leq \sum\limits_{u\in V} X^n_u(s)$
and that the sequences $X^n_v(0)$ converge implies~\eqref{eq:5v}, as claimed.
\end{proof}

\begin{lemma}
\label{D-tight}
The sequence 
$D^n$ is $\mathbb C$--tight in 
$\mathbb D(\mathbb R_+,\mathbb R^{V})$\,, i.e.
it is tight 
with all accumulation points being laws of continuous
processes.
\end{lemma}
\begin{proof}
By \eqref{eq:12},~\eqref{eq:5v} and~\eqref{Tay1}  
\begin{equation}
  \label{eq:6}
  \lim_{A\to\infty}\limsup_{n\to\infty}
\P\Big(\sup_{s\le t}| D^n_v(s)|\ge A\Big)=0\quad\text{for all}\quad v\in V.
\end{equation}
Furthermore, similarly to \eqref{Dv-bound}, we have that 
\begin{align*}
|D^n_v(s)-D^n_v(s')|
\leq e^{c_2\frac{N_n}{n}}
\int\limits_{s'}^{s}
\big(|\alpha|\Xn_v(s'')+|\beta|(\A X^n(s''))_v)\big)ds''\\
\le c_6
|s-s'|\sup_{s''\in[s',s]}
\big(|\alpha|\Xn_v(s'')+|\beta|(\A X^n(s''))_v)\big)
\end{align*}
for $s'\le s$.
Therefore, by Lemma \ref{le:X^n_bound},  for any $\epsilon>0$,
\begin{equation}
  \label{eq:7}
  \lim_{\delta\to0}\limsup_{n\to\infty}
\P\bigg(\sup_{s,s'\le t:\,|s-s'|\le\delta}|D^n_v(s)-D^n_v(s')|>
\epsilon\bigg)=0\quad\text{for all}\quad v\in V.
\end{equation}
Equations~\eqref{eq:6} and \eqref{eq:7} imply 
 that for each $v\in V$ the sequence
 $D^n_v=(D^n_v(t),t\ge0)$  is $\mathbb C$--tight in 
$\mathbb D(\mathbb R_+,\mathbb R)$, which implies the lemma.
\end{proof}

\begin{lemma}
\label{M-tight}
The sequence
$M^n$ is $\mathbb C$-tight in $\D(\R_+,\R^V)$.
\end{lemma}
\begin{proof}
For arbitrary $T>0$, 
\begin{equation}
\label{eq:9v-a}
\lim_{\delta\to0}\limsup_{n\to\infty}
\sup_{\tau\le T}\P\left(\sup_{t\le \delta}\abs{M_v^n(\tau+t)-M_v^n(\tau)}>\epsilon\right)=0
\quad\text{for all}\quad v\in V,
\end{equation}
where $\tau$ represents a stopping time. The proof of~\eqref{eq:9v-a} 
uses a bound that is analogous to \eqref{eq:4v}
  if we note that 
 by Doob's optional sampling theorem,
$(M_v^n(\tau+t)-M_v^n(\tau),t\ge0)$
 is a square integrable martingale
with predictable quadratic variation process $(
\langle M_v^n\rangle (\tau+t)-\langle M_v^n\rangle(\tau),t\ge0)$. 

Combining~\eqref{eq:9v} and~\eqref{eq:9v-a} 
with the Aldous tightness criterion (see, e.g. \cite[Theorem 6.3.1]{lipshir}),
 gives tightness of $M^n_v$ in $\mathbb D(\mathbb R_+,\mathbb R)$ for all $v\in V$.
The  $\mathbb C$-tightness of each  $M_v^n$  holds because $M^n_v$ is tight with jumps being equal to 
 $1/\sqrt n$ in absolute value, and the lemma follows.
\end{proof}

We are going to finish the proof of the theorem by
 showing $\mathbb C$-tightness of $\Xn$ and uniqueness of the limit point.

 Show first that the process 
 $\varphi^n_v(t)$ (defined in~\eqref{varphi}) tends to $0$ in probability,  as $n\to \infty$.
  To this end note  that by~\eqref{lambda<}
 $$\varphi^n_v(t)\leq e^{c_1\frac{N_n}{n}} \sqrt{n} \int_0^{t}
{\bf 1}_{\{\Xn_v(s)=N_n/\sqrt{n}\}}\,ds, 
 $$ 
where $e^{c_1\frac{N_n}{n}}\to 1$.
Fix an arbitrary $\eps>0$. Since $N_n/\sqrt{n}\to \infty$, we 
can choose  (recalling~\eqref{eq:5v}) $A<N_n/\sqrt{n}$ to be so large that
$$\limsup_n\P\left(\sup_{s\le t} X_v^n(s)>A\right)\leq \eps.
$$
Observe now that  if $\int_0^{t}
{\bf 1}_{\{\Xn_v(s)=N_n/\sqrt{n}\}}\,ds>0$, then the process exceeds the upper bound
$$
\P\bigg( \sqrt{n} \int_0^{t}
{\bf 1}_{\{\Xn_v(s)=N_n/\sqrt{n}\}}\,ds>0\bigg)\leq 
\P\left(\sup_{s\le t} X_v^n(s)>A\right)
$$
 and, hence,
$$
\limsup_{n}\P\bigg( \sqrt{n} \int_0^{t}
{\bf 1}_{\{\Xn_v(s)=N_n/\sqrt{n}\}}\,ds>0\bigg)\leq 
\eps.
$$
Since $\eps>0$ was chosen arbitrary, it follows   that 
\begin{equation}
  \label{right=0}
\limsup_{n\to\infty}\P\bigg( \sqrt{n} \int_0^{t}
{\bf 1}_{\{\Xn_v(s)=N_n/\sqrt{n}\}}\,ds>0\bigg)=0,
\end{equation}
which, in turn, implies 
that $\varphi^n_v(t)\to 0$ in probability, as $n\to \infty$.
Consequently, 
\begin{equation}
\label{t:1}
\int_0^t\lambda_v\Big(\frac{X^n(s)}{\sqrt{n}}\Big){\bf 1}_{\{\Xn_v(s)<N_n/\sqrt{n}\}}\,ds\to t,
\end{equation} 
in probability, as $n\to \infty$.

Since all the processes in
 \eqref{Xnv} except for $\phi^n_v$ are
asymptotically bounded in probability on bounded intervals (recall Lemma
\ref{le:X^n_bound} and
Lemma \ref{D-tight}),
it follows that
\[
  \lim_{A\to\infty}\limsup_{n\to\infty}
\P\big(\phi^n_v(t)>A\big)=0
\]
for any $t>0$, and, hence, owing to \eqref{eq:12c},
$\int_0^t\mathbf1_{\{X^n_v(s)=0\}}\,ds\to 0$
in probability, as $n\to \infty$, so that
\begin{equation}
\label{t:2}
\int_0^t\mathbf1_{\{X^n_v(s)>0\}}\,ds\to t.
\end{equation}
By~\eqref{eq:1}, \eqref{t:1} and \eqref{t:2},  $\langle M^n_v\rangle(t)\to 2t$
in probability, as $n\to \infty$.
Furthermore, by \eqref{eq:5}, 
$\langle M^n\rangle (t)\to 2{\mathbf I}t$, where ${\mathbf I}$ is the unit matrix. 
Since the jumps of $M^n_v$\,, 
being equal to $1/\sqrt{n}$\,, go to
zero uniformly, $M^n$ converges in distribution to 
 $\sqrt{2}B$, where $B=(B_v,\, v\in V)$ is a collection of independent
 standard Brownian motions (see e.g.
Corollary 3.24 on p.435  in  \cite{jacshir}).

By \eqref{Xnv1} 
  the process $X^n_v$ is the Skorohod reflection of the process 
$$\left(\Xn_v(0)+D^n_v(t) +M^n_v(t)-\varphi^n_v(t),\, t\ge0\right).$$
Using tightness of   $D^n$ (see Lemma~\ref{D-tight}),
 convergence of $M^n$ to $\sqrt{2}B$, convergence of $\varphi^n_v(t)$
 to $0$,
 the Lipschitz continuity of the Skorohod map $\Gamma$ 
and the continuous mapping theorem we obtain  that 
$X^n$ is $\mathbb C$--tight.
Further, if  $X=(X_v,\, v\in V)$ is a subsequential limit of $X^n$ for
convergence in distribuion, then in distribution owing to
\eqref{lambda} and  \eqref{eq:12},
 $$D^{n}_v(t)=
\sqrt{n}\int\limits_0^{t}
\bigg(\lambda_v\Big(\frac{X^{n}(s)}{\sqrt{n}}\Big)-1\bigg)ds 
\to D_v(t):=\int\limits_0^t\big(\alpha X_v(s)+\beta(\A X(s))_v\big)ds,
$$
  as $n\to \infty$, so that $X_v=\Gamma(Y_v)$, $v\in V$,  where $Y_v$, $v\in V$, follow~\eqref{Yv}.
Thus,   the process $X$ in Theorem~\ref{the:1} is the  unique limit 
 point of $X^n$,  as claimed.

\section{Stationary distribution  of the CTMC $Q(t)$ and its diffusion limit}
\label{stationary-dist}

Recall   the  adjacency matrix $\A$ 
of the graph $G=(V,E)$. Let  $(\cdot, \cdot)$ denote
  the Euclidean scalar product.
Define 
the function
\begin{equation}
   \label{W}
   W(\bx)=\frac{\alpha }{2}\sum_{v}x_v(x_v-1)+\beta \sum\limits_{v\sim u}x_vx_u
   =
   \frac{1}{2}\big((\alpha{\bf I}+\beta\A))\bx, \bx\big)-\frac{\alpha}{2}\sum_{v}x_v
\end{equation}
for $\bx=(x_v,\,v\in V)\in\R^{V}$.

\begin{lemma}
\label{Q-reversible}
The  CTMC $Q(t)$ is reversible with 
the stationary distribution given by
\begin{equation}
\label{measure}
\mu_{\alpha, \beta, N}(\bx)=\frac{e^{W(\bx)}}{\sum_{\by\in \Lambda_{N}}e^{W(\by)}},\quad
 \bx\in S_{N, V}.
\end{equation}
\end{lemma}
The lemma follows from  the detailed balance equation
\begin{equation}
\label{balance1}
r(\bx,\by)e^{W(\bx)}=e^{W\left(\by\right)}
r(\by,\bx)\quad\text{for all}\quad \bx, \by\in S_{N,V},
\end{equation}
which is the same equation as the one used  in~\cite[Section 3.1]{JVS} to show that 
the corresponding countable   CTMC  (``$N=\infty$'')
 is also   reversible with the invariant measure given by the 
 function    $e^{W(\bx)}$,  $\bx\in \Z_{+}^{V}$.
  \begin{remark}
  {\rm 
Note that in  the special case $N=1$ the change of variables $y_v=2x_v-1$ induces  
 a  probability measure on $\{-1,1\}^{V}$ which is 
 a special case of the Ising model on the  graph $G=(V,E)$.  
It was shown in~\cite[Section 4.8]{VS2023} 
that in the case $\beta>0$ the probability distribution 
$\mu_{\alpha,\beta, N}$ possesses  monotonicity properties, 
which  are similar to those of the ferromagnetic Ising model (e.g. see~\cite{GHM} and references therein).
}
\end{remark}

Further, let 
$\Xn$  be the process defined in~\eqref{Xn}. 
It follows from Lemma~\ref{Q-reversible},   that $\Xn$ is a reversible CTMC 
with the state space $S^n_{N,V}:=\{0,1/\sqrt{n},\ldots,N_n/\sqrt{n}\}^V$
and  the stationary distribution proportional to the function 
$e^{W_n(\bx^n)},\, \bx^n\in S^n_{N,V}$, 
where 
$$W_n(\bx^n)=\frac{\alpha }{2}\sum_{v}(x_v^n)^2+\beta \sum\limits_{v\sim u}x_v^nx_u^n-
\frac{\alpha}{2\sqrt{n}}\sum\limits_{v\in V}x^n_v\quad\text{for}\quad 
\bx^n=(x_v^n,\, v\in V).$$  
 If a sequence of states
$(\bx^n\in S^n_{N,V},\, n\in \N)$ converges component-wise  to 
$\bx=(x_v,\, v\in V)\in\R_{+}^V$, as  $n\to \infty$, 
then 
\begin{equation}
\label{U}
W_n(\bx^n)\to U(\bx):=
\frac{\alpha }{2}\sum_{v}x_v^2+\beta \sum\limits_{v\sim u}x_vx_u=
\frac{1}{2}((\alpha{\bf I}+\beta\A)\bx,\bx).
\end{equation}

  \begin{lemma}
\label{integral}
The integral 
\begin{equation}
\label{Z_U}
Z_U:=\int_{\R_+^{V}}e^{U(\bx)}d\bx<\infty
\end{equation}
 if and only if $\alpha<0$ and 
$\alpha+\beta\nu(G)<0$,  where $\nu(G)$ is the principal  eigenvalue of the 
 graph $G$.
\end{lemma}
The integrability criterion in the lemma is the same as a criterion of existence of the stationary 
 distribution  of the countable CTMC and can be shown by adopting the
 proof of~ \cite[Lemma 4.13]{JVS}), so we skip the details.
  Note only that if $\alpha<0$ and $\alpha+\beta\nu(G)<0$,
   then the matrix $(-\alpha{\bf I}-\beta\A)$ is 
 positive definite.  Therefore, in this  case  the 
 function $e^{U(\bx)},\, \bx\in\R_{+}^V$ is an unnormalised density of a multivariate 
 normal distribution with the zero mean and the covariance matrix $(-\alpha{\bf I}-\beta\A)^{-1}$,
 which  implies the proof of the ``if'' statement.

\begin{theorem}
\label{conv}
Let $\mu^n$ be the stationary distribution of the CTMC $X^n$. 
If $N_n/\sqrt{n}\to\infty$ and  $N_n/n\to 0$, as $n\to \infty$, then the 
 sequence $(\mu^n,\, n\in \N)$ weakly converges to the probability measure
 $\mu$ which is absolutely continuous with the 
density $e^{U(\bx)}/Z_U$, $\bx\in\R_{+}^V$,   with respect  to the Lebesgue measure on $\R_{+}^V$.
\end{theorem}
\begin{proof}
Let $f:\R_{+}^V\to \R$ be a bounded continuous function. 
Then it is easy to see that  
$$\frac{1}{n^{|V|/2}}\sum_{\bx^n\in S^n_{N,V}}f(\bx^n)e^{W_n(\bx^n)}\to
\int_{\R_{+}^V}f(\bx)e^{U(\bx)}d\bx,\quad\text{as}\quad n\to \infty.$$
Consequently, 
$$\frac{1}{n^{|V|/2}}\sum_{\bx^n\in S^n_{N,V}}e^{W_n(\bx^n)}\to
Z_U, \quad\text{as}\quad n\to \infty,$$
where $Z_U$ is defined in~\eqref{Z_U}. 
Therefore,
$$\sum_{\bx^n\in S^n_{N,V}}f(\bx^n)\mu^n(\bx^n)=
\frac{\sum_{\bx^n\in S^n_{N,V}}f(\bx^n)e^{W_n(\bx^n)}}{
\sum_{\bx^n\in S^n_{N,V}}e^{W_n(\bx^n)}}\to
\int_{\R_{+}^V}\frac{f(\bx)e^{U(\bx)}}{Z_U}d\bx=
 \int_{\R_{+}^V}f(\bx)d\mu(\bx),$$
 as $n\to \infty$.
 The theorem is proved.
\end{proof}

\begin{remark}
{\rm 
It follows from~\cite[Section 7]{JVS}
 that  
the distribution~\eqref{measure} is also
 the stationary distribution of the  CTMC  $\tilde{Q}(t)\in S_{N, V}$ 
  with the transition rates
  \begin{equation*}
\tilde r(\bx,\by)=\begin{cases}
e^{\alpha x_v},& \mbox{ for }\by=\bx+\be_v \mbox{ and }\bx=(x_v,\, v\in V): x_v<N,\\
e^{-\beta(\A\bx)_v},& 
\mbox{ for } \by=\bx-\be_v \mbox{ and } \bx=(x_v,\, v\in V): x_v>0,\\
0,& \text{otherwise,}
\end{cases}
\end{equation*}
 i.e. the transition rates, where the interaction is built into the death rate.
 
 Further, an analogue of Theorem~\ref{the:1} holds  verbatim  for the process
 $\tilde{X}^n(t)=\tilde{Q}^n(nt)/\sqrt{n}$, where $\tilde{Q}^n(t)$ is the CTMC with transition rates
 $\tilde r_n(\bx, \by)=\tilde r(\bx/n,\by/n)$. 
}
\end{remark}

\section{Open problem}
\label{problem}

Suppose that the graph $G=(V, E)$ consists of a single vertex. 
Then the limit process $X$ in Theorem~\ref{the:1} is a strong solution of the reflecting SDE 
$$dX(t)=\alpha X(t)dt+\sqrt{2}dB(t)+d\phi(t),$$
where $B$ is now  a one-dimensional standard Brownian motion,  and $\phi$ 
is a nondecreasing continuous path 
process that increases only when $X(t)=0$. 
In other words, $X$ is a one-dimensional OU process 
reflected at the origin.
This process and its multidimensional versions naturally appear 
 in queueing (see, e.g. \cite{Yamada1995},~\cite{Glynn}, and references therein). 
 The long term behaviour  is well known (see, e.g.~\cite{Glynn}). 
Namely, if $\alpha<0$,  then  the process is positive recurrent with the  stationary 
density $\frac{\sqrt{2|\alpha|}}{\sqrt{\pi}}e^{\frac{\alpha}{2}x^2}$, $x\geq 0$ 
(see Section~\ref{stationary-dist}). 
If $\alpha=0$, then the process coincides with the scaled by $\sqrt{2}$  Brownian motion 
reflected at the origin, and, hence, is
 null recurrent.
Finally, if $\alpha>0$, then the process is transient.

The problem of interest is to establish the long term behaviour of the limiting  process 
in the general case of the underlying graph. If
 $|V|\geq 2$, i.e. the number of vertices is at least $2$,
 and $\beta=0$, then the structure of the graph is irrelevant, and the 
 limit process is a collection of independent one-dimensional OU processes reflected at the origin. 
If  $\beta\neq 0$, then the interaction 
can significantly affect the collective behaviour. 
A conjecture below concerns the  long term behaviour
of the limiting process in the case when $|V|\geq 2$. 
Recall that $\nu(G)$ denotes the principal eigenvalue of the graph.
\begin{conjecture}
1) If $\alpha<0$ and 
$\alpha+\beta\nu(G)<0$, then the process $X$ is positive recurrent with the stationary distribution $\mu$
 defined in Theorem~\ref{conv}.
2) If  $\alpha=0$, $\beta\leq 0$ and $|V|=2$, then the process $X$ is null recurrent.
3) If either $\alpha>0$, or $\alpha<0$ and $\alpha+\nu(G)\beta\geq 0$, then the process $X$ is transient.
\end{conjecture}
Note that 
the conjectured long term behaviour of the process $X$
resembles that of the Markov chain
$Q$ in the countable case $N=\infty$ (see Section~\ref{model}).
 In particular,  a similar  phase transition is expected in the case
 when $\alpha<0$.
 Specifically, given $\alpha<0$,  the  process $X$ is conjectured to be 
 positive recurrent, or transient depending on whether 
  $\beta<\beta_{cr}$, or  $\beta\geq \beta_{cr}$ respectively (where $\beta_{cr}$ is defined in~\eqref{beta}).
  The main difference 
 is that the discrete process is explosive transient when $\beta>\beta_{cr}$ and non-explosive transient 
 only when $\beta=\beta_{cr}$.  In contrast, the continuous process $X$  cannot be explosive due to linearity of
 equation~\eqref{eq:20'}.

\end{document}